\documentclass{amsart}
\usepackage{amssymb,amscd,MnSymbol}
\usepackage{mathrsfs}
\usepackage{graphicx,psfrag}

\usepackage[initials]{amsrefs}

\newcommand{\si}{\sigma}
\newcommand{\Si}{\Sigma}

\newcommand{\bC}{\mathbb{C}}

\newcommand{\bF}{\mathbb{F}}
\newcommand{\bK}{\mathbb{K}}
\newcommand{\bL}{\mathbb{L}}
\newcommand{\bP}{\mathbb{P}}
\newcommand{\bQ}{\mathbb{Q}}
\newcommand{\bR}{\mathbb{R}}
\newcommand{\bT}{\mathbb{T}}
\newcommand{\bZ}{\mathbb{Z}}

\newcommand{\bSi}{{\mathbf\Si}}

\newcommand{\bGamma}{{\boldsymbol{\Gamma}}}

\newcommand{\cA}{\mathcal{A}}

\newcommand{\cD}{\mathcal{D}}
\newcommand{\cE}{\mathcal{E}}

\newcommand{\cF}{\mathcal{F}}
\newcommand{\cI}{\mathcal{I}}
\newcommand{\cL}{\mathcal{L}}
\newcommand{\cM}{\mathcal{M}}
\newcommand{\cO}{\mathcal{O}}

\newcommand{\cT}{\mathbb{T}}
\newcommand{\cU}{\mathcal{U}}
\newcommand{\cV}{\mathcal{V}}
\newcommand{\cX}{\mathcal{X}}
\newcommand{\cIX}{\mathcal{IX}}
\newcommand{\cY}{\mathcal{Y}}
\newcommand{\cZ}{\mathcal{Z}}

\newcommand{\fp}{\mathfrak{p}}

\newcommand{\Fuk}{\mathrm{Fuk}}
\newcommand{\Hom}{\mathrm{Hom}}
\newcommand{\Coh}{\mathrm{Coh}}
\newcommand{\age}{\mathrm{age}}
\newcommand{\eff}{{\mathrm{eff}}}

\newcommand{\SYZ}{{\mathrm{SYZ}}}
\newcommand{\orb}{\mathrm{CR}}
\newcommand{\pr}{\mathrm{pr}}
\newcommand{\inv}{\mathrm{inv}}
\newcommand{\ev}{\mathrm{ev}}

\newcommand{\Pic}{ {\mathrm{Pic}} }

\newcommand{\Ker}{\mathrm{Ker}}

\newcommand{\Sh}{\mathrm{Sh}}
\newcommand{\Spec}{\mathrm{Spec}}

\newcommand{\vir}{{\mathrm{vir}}}
\newcommand{\CC}{\mathrm{CC}}
\newcommand{\CR}{  {\mathrm{CR}}  }
\newcommand{\Nef}{ {\mathrm{Nef}} }
\newcommand{\NE}{ {\mathrm{NE}} }
\newcommand{\BoxS}{\mathrm{Box}(\mathbf \Si)}
\newcommand{\Boxs}{\mathrm{Box}(\si)}

\newcommand{\one}{\mathbf{1}}

\newcommand{\btau}{{\boldsymbol\tau }}

\newcommand{\tC}{ {\widetilde{C}} }
\newcommand{\tbT}{ {\widetilde{\bT}} }
\newcommand{\tcD}{ {\widetilde{\mathcal{D}}}}
\newcommand{\tW}{{\widetilde{W}}}
\newcommand{\tb}{{\widetilde{b}}}

\newcommand{\tM}{{\widetilde{M}}}
\newcommand{\tN}{{\widetilde{N}}}

\newcommand{\tV}{{\widetilde{V}}}

\newcommand{\tw}{\widetilde{w}}
\newcommand{\tuw}{\widetilde{\uw}}

\newcommand{\tcY}{\widetilde{\cY}}

\newcommand{\tGamma}{{\widetilde{\Gamma}}}

\newcommand{\tch}{ {\widetilde{\mathrm{ch}}}}
\newcommand{\tNef}{\widetilde{\Nef}}
\newcommand{\tNE}{\widetilde{\NE}}

\newcommand{\vl}{{\vec{l}}}

\newcommand{\IX}{\mathcal{IX}}
\newcommand{\uchi}{{\underline{\chi}}}

\newcommand{\lra}{\longrightarrow}
\newcommand{\pt}{\mathrm{point}}
\newcommand{\uw}{\mathsf{w}}
\renewcommand{\SS}{\mathrm{SS}}

\newtheorem*{Theorem*}{Theorem}
\newtheorem{dummy}{dummy}[section]
\newtheorem{lemma}[dummy]{Lemma}
\newtheorem{theorem}[dummy]{Theorem}

\newtheorem{proposition}[dummy]{Proposition}

\theoremstyle{definition}
\newtheorem{definition}[dummy]{Definition}
\newtheorem{remark}[dummy]{Remark}
\newtheorem{example}[dummy]{Example}
\newtheorem{assumption}[dummy]{Assumption}

\begin{document}
\title{Central charges of T-dual branes for toric varieties}

\address{Bohan Fang, Beijing International Center for Mathematical
  Research, Peking University, 5 Yiheyuan Road, Beijing 100871, China}
\email{bohanfang@gmail.com}
\author{Bohan Fang}
\begin{abstract}
Given any equivariant coherent sheaf $\cL$ on a compact semi-positive
toric orbifold $\cX$, its SYZ T-dual mirror dual is a Lagrangian brane
in the Landau-Ginzburg mirror. We prove the oscillatory integral of
the equivariant superpotential in the Landau
Ginzburg mirror over this Lagrangian brane is the genus-zero $1$-descendant Gromov-Witten potential with a Gamma-type class of
$\cL$ inserted.
\end{abstract}
\maketitle

\section{Introduction}

Mirror symmetry relates Gromov-Witten invariants
to more classical objects such as period integrals. For an
$n$-dimensional toric orbifold $\cX$, its mirror is a Landau-Ginzburg
model $W:(\bC^*)^{n}\to \bC$, where the superpotential $W$ is a
Laurent polynomial.

There are many ways to extract Gromov-Witten invariants from this
B-model, like quantum cohomology \cites{FOOO10,FOOO11,TsWa} and genus-zero descendant
invariants \cites{Gi96a,Gi96b,Gi97,LLY97,LLY99a,LLY99b,CCIT15, CCK15}. In this
paper, we would be particularly interested in what the following integration says about the A-model Gromov-Witten theory:
\begin{equation}
\label{eqn:integral-introduction}
\int_{\Xi} e^{-\frac{W}{z}} \Omega, \quad \Omega=\frac{dX_1}{X_1}\dots\frac{dX_n}{X_n}.
\end{equation}

\subsection{Motivation}
This paper is motivated from several aspects.
\subsubsection{SYZ T-duality and homological mirror symmetry}

Kontsevich's homological mirror symmetry \cite{Kon94,Kon97} relates the bounded derived category of coherent sheaves on a complete toric variety (or orbifold) $\cX$ to the derived Fukaya-Seidel type category of its mirror Landau-Ginzburg model. There is a correspondence
\[
D\mathrm{FS}((\bC^*)^n,W)\cong D^b\Coh(\cX).
\]
Here $\mathrm{FS}$ is a kind of Fukaya-Seidel category, and $\Coh$ is the category of coherent sheaves. This statement is proved using different setups of Fukaya-Seidel categories \cite{Ab09,FLTZ12}.

There are more structures in these categories of branes. There should be a \emph{BPS central charge} or simply \emph{central charge} associated to each brane, either in A or B-model. These central charges play important roles in the stability conditions of Bridgeland \cite{Br07}.

The canonical central charge of a Lagrangian submanifold $\Xi$ in a Calabi-Yau manifold (an A-brane) is the integration of the Calabi-Yau form
\[
\int_{\Xi} \Omega.
\]
The choice of the central charge reflects the \emph{complex} structure of the Calabi-Yau manifold, manifested by the Calabi-Yau form here. On the mirror side the central charge of a B-model brane a.k.a. a coherent sheaf is part of the stability condition which in principle should specify the stringy K\"ahler structure. The definition should involve the quantum correction by Gromov-Witten invariants. In \cite{Ho06}, Hosono prescribes the central charges formula for certain Calabi-Yaus. The formula for a B-brane involves a hypergeometric series which comes from Gamma classes \cite{Ir09,KKP} and mirror symmetry's $I$-function. These central charges for the A and B-branes should be equal under mirror symmetry.

In this paper we will look at the central charges for the mirror pair of a complete semi-positive toric orbifold and its mirror Landau-Ginzburg model. We express the central charge of a brane in $D^b\Coh(\cX)$ as a genus-zero descendant Gromov-Witten potential as in \cite{Ir09}, and prove it is equal to the integral \eqref{eqn:integral-introduction} on the mirror Lagrangian brane in the Landau-Ginzburg model (central charge on the Lagrangian brane).

\subsubsection{Integral structures}

Iritani extensively investigates the oscillatory integral \eqref{eqn:integral-introduction} in \cite{Ir09}. The integral cycles $\Xi\in
H_n(X,\mathrm{Re}(W)\gg 0;\bZ)$ form a B-model
\emph{integral} structure. The genus-zero
Gromov-Witten defines the quantum cohomology $(H_\CR^*(\cX),
\circ_{\btau})$. The quantum D-module is a flat connection on the
trival bundle $H_\CR^*(\cX)\times H_\CR^*(\cX)\to H_\CR^*(\cX)$ with a
parameter $z\in \bC^*$, called the \emph{Dubrovin connection}
\[
\nabla_\alpha=\partial_\alpha+\frac{1}{z}\alpha  \circ_\btau,\quad \alpha\in H_\CR^*(\cX).
\]
This connection can be extended to a flat connection on the trivial
$H^*_\CR(\cX)$-bundle on $H^*_\CR(\cX)\times \bC^*$.
For any $V\in K(\cX)$, one associates a flat section
$\cZ(V)(\btau,z)$. All such $\cZ(V)$ define an
integral lattice in the space of flat sections of $\nabla$.
The B-model D-module and its integral structure is naturally generated by the lattice
$H_n(\cY_q,\mathrm{Re}(W/z)\gg 0;\bZ)$, where $\cY_q\cong (\bC^*)^{\dim\cX}$
is the mirror Landau-Ginzburg model of $\cX$, parametrized by the
complex parameter $q$. The identification of the D-modules also
identifies such integral structures \cite{Ir09}.

As noted in \cite{Ir09}, the correspondence should come from
homological mirror symmetry. Indeed, Iritani
proves the identification of the integral structures by first studying
the case when $V$ is either $\cO_\cX$ or the skyscraper sheaf. In
either case, the corresponding cycle $\Xi$ in Equation
\eqref{eqn:integral-introduction} is the Strominger-Yau-Zaslow dual
to $V$ \cite{AP,LYZ,SYZ}, i.e. cotangent fiber (viewing $\cY_q\cong (\bC^*)^n=T^*(S^1)^n$) or the base $(S^1)^n\subset \cY_q$. This identification follows the theme of homological
mirror symmetry for toric varieties \cite{Ab09,FLTZ11,FLTZ12,FLTZ14}. Let $Z(V)$ be the
integration of $\cZ(V)$ over $\cX$ (pairing with $\cX$'s fundamental
class). The value of $Z(V)$ is directly computed as a descendant
potential
\[
Z(V)=\llangle\frac{\bGamma_z(V)}{z(z+\psi)}\rrangle_{0,1}^\cX,
\]
where $\bGamma_z$ is a transcendental characteristic class involving the Gamma
function. On the B-side, for example when $V$ is the structure sheaf $\cO_\cX$, Iritani obtains the following under the mirror map when $\btau \in H^{\leq 2}_{\orb}(\cX;\bC)$
\[
Z(\cO_\cX)=\int_{\Xi(O_\cX)} e^{-W/z} \Omega.
\]
Here SYZ dual $\Xi(\cO_\cX)$ is $(\bR_{>0})^n\subset
\cY_q\cong(\bC^*)^n$. Homological mirror symmetry says
$\Xi(\cO_\cX)$ is the Lagrangian brane in the Landau-Ginzburg model
mirror to the coherent sheaf $\cO_\cX$ in $D^b\Coh(\cX)$. Moreover, Iritani also matches Galois actions on A and B-sides to generate and match all integral structures, i.e. two ``tensoring a line bundle'' operations on both sides.

In this paper, we develop this correspondence of central charges more carefully. We show
an equivariant version of the above correspondence with an explicit
description of $\Xi(\cL)$ for any $\cL\in D^b\Coh_\bT(\cX)$, the derived category of equivariant coherent sheaves on $\cX$. \footnote{For the cycle mirror to the structure sheaf, Iritani computes $Z(\cO_\cX)$ equivariantly first and then takes non-equivariant limit.}  The
explicit description of $\Xi(\cL)$ is the same as the Lagrangian
brane obtained in the theorem of homological mirror symmetry of
\cite{FLTZ12}.
\subsection{Main theorem}
\begin{Theorem*}
Let $\cX$ be a proper semi-positive toric orbifold, $\tW$ be the
equivariantly perturbed superpotential, which is a holomorphic
function on the universal cover $\tcY_q$ of $\cY_q$ . Under the mirror
map $\btau=\btau(q_0,q)$ (as in \cite{Gi96a,Gi96b,Gi97,LLY97,LLY99a,LLY99b})
\[
\llangle \frac{\kappa(\cL)}{z(z+\psi)}\rrangle_{0,1}^\cX=
\int_{\Xi(\cL)} e^{-\frac{\tW}{z}}\Omega.\]
The symbol $\llangle \rrangle_{0,1}^{\cX}$ is the notation of genus-zero primary and descendant Gromov-Witten potential with $1$ marked point (see Definition \ref{def:gw-potential}). Here $\Xi(\cL)$ is a piecewise linear conical Lagrangian cycle in $\tcY_q$, and is the conical limit of the \emph{equivariant} SYZ T-dual brane of
$\cL$.\footnote{A conical set is a dilation-invariant set in certain
  direction. See Definition \ref{def:conical}.}
\end{Theorem*}
\begin{remark}
  The class $\kappa(\cL)$ related to an equivariant version of the Gamma class \eqref{eqn:gamma}.
\end{remark}

\begin{remark}
  \label{rmk:Gamma}
  This theorem should have some implications, at least in principle, on the \emph{Gamma II conjecture} proposed by Galkin-Golyshev-Iritani \cite{GGI16} for semi-positive complete toric orbifolds. A proof of the K-theoretic version is in \cite[p19]{GaIr15}. Such conjecture states when the non-equivariant quantum cohomology is semisimple, like in the complete toric case, the fundamental solutions to the QDEs have certain asymptotic behavior prescribed by the non-equivariant limit of $\kappa$-classes of coherent sheaves in an exceptional collection of $D^b\Coh(\cX)$. The existence of the exceptional collection in the toric case is known \cite{Kaw06}, while recent works \cite{Kuw16, Vai16} allow us to take non-equivariant limit for  coherent-constructible correspondences (see also \cites{ScSi16,Tre10,Kuw15}).
\end{remark}

\begin{remark}
One can simply integrate over the SYZ T-dual brane $\SYZ(\cL)$ instead of
$\Xi(\cL)$ and get the same result since it is
homotopic to and has
the same asymptotic behavior as $\Xi(\cL)$ (see Remark \ref{rem:SYZ}).
\end{remark}
\begin{remark}
In \cite{FLZ15}, Liu, Zong and the author directly compute the
oscillatory integrals when $\cX$ is $\bP^1$. Combining with the
Eynard-Orantin recursion, one can obtain all genus descendant
potentials of $\bP^1$ by integrating the Eynard-Orantin higher genus
invariants over the suitable cycles.
\end{remark}

\subsection{Outline}
We fix the notion of compact semi-positive toric orbifolds in Section
\ref{sec:toric}, and recall some basic facts of homological mirror
symmetry for complete toric orbifolds in Section \ref{sec:hms}. In
Section \ref{sec:oscillatory} we compute the oscillatory integrals by fixing the
coefficients due to the fact that the integral is a solution to the GKZ system. We
prove the main theorem in Section \ref{sec:mirror} by invoking the genus-zero
mirror theorem of Givental and Lian-Liu-Yau  \cite{Gi96a,Gi96b,Gi97,LLY97,LLY99a,LLY99b}.

\subsection{Acknowledgement}
The author would like to thank Chiu-Chu Melissa Liu and Zhengyu Zong
for valuable discussion. The author would also like to thank Hiroshi Iritani and Eric Zaslow for very helpful comments. This work is partially supported by the
Recruitment Program of Global Experts in China and a start-up grant at
Peking University.

\section{Smooth Toric DM Stacks}
\label{sec:toric}
In this section, we follow the definitions in \cite[Section 3.1]{Ir09}, with slightly different notation. We work over $\bC$.

\subsection{Definition}
Let $N\cong \bZ^n$ be a finitely generated free abelian group, and let $N_\bR=N\otimes_\bZ\bR$. We consider complete toric \emph{orbifolds} which have trivial generic stabilizers. A toric orbifold is given by the following data:
\begin{itemize}
\item vectors $b_1,\dots, b_{r'} \in N$. We require the subgroup $\oplus_{i=1}^{r'} \bZ b_i$ is of finite index in $N$.
\item a complete simplicial fan $\Si$ in $N_\bR$ such that the set of $1$-cones is
$$
\{\rho_1,\dots,\rho_{r'}\},
$$
where $\rho_i=\bR_{\ge 0} b_i$, $i=1,\dots, r'$.
\end{itemize}
The datum $\mathbf \Si=(\Si, (b_1,\dots, b_{r'}))$ is the \emph{stacky
  fan} \cite{BCS05}. We choose
\emph{additional} vectors  $b_{r'+1},\dots, b_r$ such that $b_1,\dots, b_r$ generate $N$.
There is a surjective group homomorphism
\begin{eqnarray*}
\phi: & \tN :=\oplus_{i=1}^r \bZ\tb_i & \lra  N,\\
        &  \tb_i  & \mapsto b_i.
\end{eqnarray*}
Define $\bL :=\Ker(\phi) \cong \bZ^k$, where $k:=r-n$. Then
we have the following short exact sequence of finitely generated abelian groups:
\begin{equation}\label{eqn:NtN}
0\to \bL  \stackrel{\psi }{\lra} \tN  \stackrel{\phi}{\lra} N\to 0.
\end{equation}
Applying $ - \otimes_\bZ \bC^*$ and $\Hom(-,\bZ)$ to \eqref{eqn:NtN}, we obtain two exact
sequences of abelian groups:
\begin{align}
\label{eqn:bT}
&1 \to G \to \tbT\to \bT \to 1,\\
&\label{eqn:MtM}
0 \to M \stackrel{\phi^\vee}{\to} \tM \stackrel{\psi^\vee}{\to} \bL^\vee \to 0,
\end{align}
where
\begin{align*}
&\bT= N\otimes_\bZ \bC^* ={N}\otimes_\bZ \bC^* \cong (\bC^*)^n,\  \tbT = \tN\otimes_\bZ \bC^* \cong (\bC^*)^r,\  G = \bL\otimes_\bZ \bC^* \cong (\bC^*)^k,\\
&M = \Hom(N,\bZ)  = \Hom(\bT,\bC^*), \
\tM = \Hom(\tN,\bZ)= \Hom(\tbT,\bC^*),\
\bL^\vee = \Hom(\bL,\bZ) =\Hom(G,\bC^*).
\end{align*}

The action of $\tbT$ on itself extends to a $\tbT$-action on $\bC^r = \Spec\bC[Z_1,\dots, Z_r]$.
The group $G$ acts on $\bC^r$ via the group homomorphism $G\to \tbT$ in \eqref{eqn:bT}.

Define the set of ``anti-cones''
$$
\cA=\{I\subset \{1,\dots, r\}: \text{$\sum_{i\notin I} \bR_{\ge 0} b_i$ is a cone of $\Si$}\}.
$$
Given $I\in \cA$, let $\bC^I$ be the subvariety of $\bC^r$ defined by the ideal in $\bC[Z_1,\ldots, Z_r]$ generated by $\{ Z_i \mid i\in I\}$.  Define the toric orbifold $\cX$ as the stack quotient
$$
\cX:=[U_\cA/ G],
$$
where
$$
U_\cA:=\bC^r \backslash \bigcup_{I\notin \mathcal A} \bC^I.
$$
$\cX$ contains the torus $\cT:= \tbT/G$ as a dense open subset, and
the $\tbT$-action on $\cU_\Si$ descends to a $\cT$-action on $\cX$.

\subsection{Line bundles and divisors on $\cX$}

Let $\tcD_i$ be the $\tbT$-divisor in $\bC^r$ defined by $Z_i=0$. Then
$\tcD_i \cap \cU_A$ descends to a $\bT$-divisor $\cD_i$ in $\cX$. We have
$$
\tM \cong \Pic_{\tbT}(\bC^r) \cong H^2_{\tbT}(\bC^r;\bZ),
$$
where the second isomorphism is given by the $\tbT$-equivariant
first Chern class $(c_1)_{\tbT}$.  Define\footnote{The minus sign in the definition of $\uw_i$ is a convention for the purporse of matching the B-model oscillatory integrals. See Section \ref{sec:oscillatory}.}
\begin{align*}
\uw_i& = -(c_1)_{\tbT}(\cO_{\bC^r}(\tcD_i)) \in H^2_{\tbT}(\bC^r;\bZ) \cong H^2_{\cT}([\bC^r/G];\bZ)\cong\tM,\\
\bar D_i^{\bT}&=(c_1)_{\bT}(\cO_{\cX}(\cD_i)),\quad \bar D_i^{\tbT}=(c_1)_{\tbT}(\cO_{\cX}(\cD_i)).
\end{align*}
Then $\{-\uw_1,\ldots,-\uw_r\}$ is a $\bZ$-basis of
$H^2_{\tbT}(\bC^r;\bZ)\cong \tM$ dual to the $\bZ$-basis
$\{ \tb_1,\ldots, \tb_r\}$ of $\tN$. Notice that $
\bar{D}_i^\cT =0$ and $\bar{D}_i^{\tbT}=0$ for $i=r'+1,\ldots, r$. We have
\begin{gather*}
\Pic_{\cT}(\cX)\cong H^2_{\cT}(\cX;\bZ) \cong \tM/\oplus_{i=r'+1}^r \bZ D^\cT_i,\\
H^2_{\cT}(\cX;\bZ) = \bigoplus_{i=1}^{r'} \bZ \bar{D}^\cT_i \cong\bZ^{r'}.
\end{gather*}
We also have group isomorphisms
$$
\bL^\vee \cong \Pic_G(\bC^r) \cong H^2_G(\bC^r;\bZ),
$$
where the second isomorphism is given by the $G$-equivariant first Chern class $(c_1)_G$. Notice that a character $\chi \in\bL^\vee$ defines a line bundle on $\cX$
$$
\cL_\chi=\{(z,t)\in \cU_A\times \bC/ (z,t)\sim (g\cdot z,\chi(g)\cdot t),g\in G\},
$$
where $G$ acts on $\cU_A$ as a subgroup of $\tbT$. There is a canonical $\tbT$-action on $\cL_\chi$ which acts diagonally on $\bC^r$ and trivially on the fiber. Define
\begin{gather*}
D_i=(c_1)_G(\cO_{\bC^r}(\tcD_i)) \in H^2_G(\bC^r;\bZ)\cong \bL^\vee,\\
\bar{D}_i=c_1(\cO_{\cX}(\cD_i))\in H^2(\cX;\bZ).
\end{gather*}
We have
$$
\Pic(\cX)\cong H^2(\cX;\bZ) \cong \bL^\vee/\oplus_{i=r'+1}^r \bZ D_i.
$$
The map
$$
\bar{\psi}^\vee: \Pic_{\cT}(\cX)\cong H^2_{\cT}(\cX;\bZ) \to \Pic(\cX)\cong H^2(\cX;\bZ)
$$ is
descends from $\psi^\vee:\Pic_\tbT(\bC^r;\bZ)\cong \tM \to \Pic_G(\bC^r;\bZ)\cong \bL^\vee$ and satisfies
$$
\bar \psi^\vee(\bar{D}_i^\cT)=\bar{D}_i\quad i=1,\ldots, r'.
$$

\subsection{Torus invariant subvarieties and their generic stabilizers}
Let $\Si(d)$ be the set of $d$-dimensional cones. For each $\si\in \Si(d)$,
define
$$
I_\si=\{ i\in \{1,\ldots,r\}\mid \rho_i\not\subset \si \} \in \cA,
\quad I_\si'= \{1,\ldots, r\}\setminus I_\si.
$$
Let $\tilde{V}(\si)\subset U_\cA$ be the closed subvariety defined by the ideal of $\bC[Z_1,\ldots, Z_r]$ generated by
$$
\{Z_i=0\mid \rho_i\subset \si\} = \{ Z_i=0\mid i \in I'_\si\}.
$$
Then $\cV(\si) := [\tV(\si)/G]$ is an $(n-d)$-dimensional $\cT$-invariant closed
subvariety of $\cX= [U_\cA/G]$.

Let
$$
G_\si:= \{ g\in G\mid g\cdot z = z \textup{ for all } z\in \tV(\si)\}. 
$$
Then $G_\si$ is the generic stabilizer of $\cV(\si)$. It is a finite subgroup of $G$.
If $\tau\subset \si$ then $I_\si\subset I_\tau$, so $G_\tau\subset G_\si$. There
are two special cases. If $\si\in \Si(n)$ where $n=\dim_\bC \cX$, then $\fp_\si:= \cV(\si)$ is a $\cT$ fixed point in $\cX$, and
$\fp_\si=BG_\si$.

\subsection{The extended nef cone and the extended Mori cone} \label{sec:nef-NE}
In this paragraph, $\bF=\bQ$, $\bR$, or $\bC$.
Given a finitely generated free abelian group $\Lambda\cong \bZ^m$, define
$\Lambda_\bF= \Lambda\otimes_\bZ \bF \cong \bF^m$.
We have the following short exact sequences of vector spaces ($\otimes
\bF$ with Equation \eqref{eqn:NtN} and \eqref{eqn:MtM}):
\begin{eqnarray*}
&& 0\to \bL_\bF\to \tN_\bF \to N_\bF \to 0,\\
&& 0\to  M_\bF\to \tM_\bF\to \bL^\vee_\bF\to 0.
\end{eqnarray*}

Given a maximal cone $\si\in \Si(n)$, we define
$$
\bK_\si^\vee := \bigoplus_{i\in I_\si }\bZ D_i.
$$
Then $\bK_\si^\vee$ is a sublattice of $\bL^\vee$ of finite index.
We define the {\em extended nef cone} $\tNef_\cX$ as below
$$
\tNef_\si = \sum_{i\in I_\si}\bR_{\geq 0} D_i,\quad \tNef_{\cX}:=\bigcap_{\si\in \Si(n)} \tNef_\si.
$$

The {\em extended $\si$-K\"{a}hler cone} $\tC_\si$ is defined to be the interior of
$\tNef_\si$; the {\em extended K\"{a}hler cone} of $\cX$, $\tC_{\cX}$, is defined
to be the interior of the extended nef cone $\tNef_{\cX}$.

Let $\bK_\si $ be the dual lattice of $\bK_\si^\vee$; it can be viewed as an additive subgroup of $\bL_\bQ$:
$$
\bK_\si =\{ \beta\in \bL_\bQ \mid \langle D, \beta\rangle \in \bZ \  \forall D\in \bK_\si^\vee \},
$$
where $\langle-, -\rangle$ is the natural pairing between
$\bL^\vee_\bQ$ and $\bL_\bQ$. We have $\bK_\si/\bL\cong G$. Define
$$
\bK:= \bigcup_{\si\in \Si(n)} \bK_\si.
$$
Then $\bK$ is a subset (which is not necessarily a subgroup) of $\bL_\bQ$, and $\bL\subset \bK$.

We define the {\em extended Mori cone} $\tNE_\si\subset \bL_\bR$ to be
$$
\tNE_{\cX}:= \bigcup_{\si\in \Si(n)} \tNE_\si,\quad  \tNE_\si=\{ \beta \in \bL_\bR\mid \langle D,\beta\rangle \geq 0 \ \forall D\in \tNef_\si\}.
$$
Finally, we define extended curve classes
$$
\bK_{\eff,\si}:= \bK_\si\cap \tNE_\si,\quad \bK_{\eff}:=
\bK\cap \tNE(\cX)= \bigcup_{\si\in \Si(n)} \bK_{\eff,\si}.
$$

\begin{assumption}[semi-positive (Weakly Fano) condition] \label{semi-positive}
From now on, we assume that we may choose $b_{r'+1},\ldots, b_r$ such that
$\hat \rho:=D_1+\dots + D_r$ is contained in the closure
of the extended K\"ahler cone $\tC_\cX$.
\end{assumption}

\begin{remark}
We make the above assumption  so that the equivariant mirror theorem in \cite{CCIT15,CCK15}
has an explicit mirror map.
\end{remark}

\subsection{The inertia stack and the Chen-Ruan cohomology} \label{sec:CR}
Given $\si\in \Si$, define
$$
\Boxs:=\{ v\in N: {v}=\sum_{i\in I'_\si} c_i {b}_i, \quad 0\leq c_i <1\}.
$$
If $\tau\subset \sigma$ then $I'_\tau\subset I'_\si$, so $\mathrm{Box}(\tau)\subset \Boxs$.


Given a real number $x$, let $\lfloor x \rfloor$ be the greatest integer less than or equal to  $x$,
$\lceil x \rceil$ be the least integer greater or equal to $x$,
and $\{ x\} = x-\lfloor x \rfloor$ is the fractional part of $x$.
Define $v: \bK_\si\to N$ by
\begin{equation}
  \label{eqn:v}
v(\beta)= \sum_{i=1}^r \lceil \langle D_i,\beta\rangle\rceil b_i=\sum_{i\in I'_\si} \{ -\langle D_i,\beta\rangle \}b_i,
\end{equation}
so $v(\beta)\in \Boxs$. Indeed, $v$ induces a bijection $\bK_\si/\bL\cong \Boxs$. 

For any $\tau\in \Si$ there exists $\si\in \Si(n)$ such that
$\tau\subset \si$. The bijection $G_\si \to \Boxs$ restricts
to a bijection $G_\tau\to \mathrm{Box}(\tau)$. Define
$$
\BoxS:=\bigcup_{\si\in \Si}\Boxs =\bigcup_{\si\in\Si(n)}\Boxs.
$$
There is a bijection $\bK/\bL\to \BoxS$. 

Given $v\in \Boxs$, where $\si\in \Si(d)$, define $c_i(v)\in [0,1)\cap \bQ$ by
$$
v= \sum_{i\in I'_\si} c_i(v) b_i.
$$
Suppose that  $k \in G_\si$  corresponds to $v\in \Boxs$ under the
bijection $G_\si\cong\Boxs$, then
$$
\chi_i(k) = \begin{cases}
1, & i\in I_\si,\\
e^{2\pi\sqrt{-1} c_i(v)},& i \in I'_\si.
\end{cases}
$$
Define
$$
\age(k)=\age(v)= \sum_{i\notin I_\si} c_i(v).
$$

Let $IU=\{(z,k)\in U_\cA\times G\mid k\cdot z = z\}$,
and let $G$ acts on $IU$ by $h\cdot(z,k)= (h\cdot z,k)$. The
inertia stack $\cI\cX$ of $\cX$ is defined to be the quotient stack
$$
\cI\cX:= [IU/G].
$$
The inertial stack $\cIX$ comes with a projection map $\pr:\cIX\to \cX$. Note that $(z=(Z_1,\ldots,Z_r), k)\in IU$ if and only if
$$
k\in \bigcup_{\si\in \Si}G_\si \textup{ and }  Z_i=0 \textup{ whenever } \chi_i(k) \neq 1.
$$
So
$$
IU=\bigcup_{v\in \BoxS} U_v,
$$
where
$$
U_v:= \{(Z_1,\ldots, Z_m)\in U_\cA: Z_i=0 \textup{ if } c_i(v) \neq 0\}.
$$
The connected components of $\cI\cX$ are
$$
\{ \cX_v:= [U_v/G] : v\in \BoxS\}.
$$
The involution $IU\to IU$, $(z,k)\mapsto (z,k^{-1})$ induces involutions
$\inv:\cI\cX\to \cI\cX$ and $\inv:\BoxS\to \BoxS$ such that
$\inv(\cX_v)=\cX_{\inv(v)}$. Define the $\bT$-fixed point $\fp_{\si,v}=\pr^{-1}(\fp_\si)\cap \cX_v$.

In the remainder of this subsection, we consider rational cohomology, and
write $H^*(-)$ instead of $H^*(-;\bQ)$. The Chen-Ruan orbifold cohomology, as a vector space, is defined to be \cite{Zas93, CR04}
$$
H^*_\orb (\cX)=\bigoplus_{v\in \BoxS}  H^*(\cX_v)[2\age(v)].
$$
Denote $\mathbf 1_v$ to be the unit in $H^*(\cX_v)$. Then $\mathbf 1_v\in H^{2\age(v)}_\orb(\cX)$. In particular,
$$
H^0_\orb(\cX) =  \bQ \mathbf 1_0.
$$
Since $\cX$ is proper, the orbifold Poincar\'e pairing on $H^*_\orb(\cX)$ is defined as
\begin{equation}\label{eqn:Poincare}
(\alpha,\beta):=\int_{\cI\cX} \alpha\cdot \inv^*(\beta),
\end{equation}
We also have an equivariant pairing on $H^*_{\orb,\bT}(\cX)$:
\begin{equation}\label{eqn:T-Poincare}
(\alpha,\beta)_{\bT} := \int_{\cI\cX_{\bT}} \alpha\cdot \inv^*(\beta),
\end{equation}
where
$$
\int_{ \cI\cX_{\bT}}: H_{\orb,\bT}^*(\cX) \to H_{\bT}^*(\pt) = H^*(B\bT)
$$
is the equivariant pushforward to a point. Here the dot $\cdot$ (sometimes omitted) is the usual product in $H^*(\cIX)$ as the same vector space as $H^*_\orb(\cX)$. The product in the Chen-Ruan cohomology, as defined in \cite{CR04} or for many cases \cite{Zas93}, does not explicitly appear in this paper.


\section{The mirror of a toric orbifold}
\label{sec:hms}
\subsection{Coherent-constructible correspondence and HMS: a quick review}

The equivariant mirror to $\cX$ is a Landau-Ginzburg model $((\bC^*)^n,W)$, where $W:(\bC^*)^n\to \bC$ is the holomorphic superpotential function. To consider the \emph{equivariant} information on $\cX$, one needs to consider the \emph{universal cover} of $(\bC^*)^n$. As a symplectic manifold, this $(\bC^*)^n$ is naturally identified with $\cY=T^*(M_\bR/M)\cong T^*(S^1)^n$, and its universal cover $\bC^n$ is identified with $\tcY=T^*M_\bR=N_\bR\times M_\bR$.

For any $\tau \in \Si$, let $N_\tau$ be the sublattice of $N$
generated by $b_i\in \tau$, and $M_\tau=\Hom(N_\tau,\bZ)$. Define
$\tau^\perp_{\uchi}=\{u\in M_\bR|\langle v,u \rangle=\uchi(v), v\in \tau\}$.
\begin{definition}
\label{def:conical}
Let $U$ be a real manifold. We say a set in $T^*U$ is conical if it is invariant under the
dilation $(v,u)\mapsto (t v,u)$ where $u\in U$, $v\in T^*_vU$ and
$t\in \bR_{>0}$.
\end{definition}
The stacky fan $\bSi$ defines the following conical Lagrangian in $T^*M_\bR=N_\bR\times M_\bR$.
$$
\Lambda_\bSi=\bigcup_{\tau\in\Si} \bigcup_{\uchi\in M_\tau}(-\tau)\times \tau^\perp_{\uchi},
$$
We embed $\rho: T^*M_\bR \hookrightarrow \overline {T^* M_\bR}=\{(v,u)\in
T^*M_\bR|\|v\|\leq 1\}$ by $$(v,u)\mapsto
(\frac{v}{\sqrt{1+\|v\|^2}},u).$$ Here $v\in N_\bR$ and $u\in M_\bR$. The infinity part of $\Lambda_\bSi^\infty=\overline {\Lambda_\bSi}\backslash \Lambda_\bSi$ is piecewise Legendrian in the contact boundary unit sphere bundle $T^\infty M_\bR=\overline{T^*M_\bR} \backslash T^*M_\bR$.

We refer the reader to \cite{FLTZ11,FLTZ12,FLTZ14} for the details of
homological mirror symmetry, and only recall the relevant parts
here. Let $D^b\Coh_\bT(\cX)$ be the derived category of equivariant
coherent sheaves on $\cX$, and let $D^b\Sh_{cc}(M_\bR;\Lambda_\bSi)$ be the
derived category of compactly supported constructible sheaves on
$M_\bR$ whose singular supports are subsets of $\Lambda_\bSi$. Define
$\Fuk(T^*M_\bR;\Lambda_\bSi)$ be the derived zunwrapped Fukaya category (in
the sense of \cite{NZ09,Na09}) of $T^*M_\bR$ of horizontally compactly
supported Lagrangian branes $L$ such that their infinity parts
$L^\infty\subset \Lambda_\bSi^\infty$. The homological mirror symmetry of Fang-Liu-Treumann-Zaslow \cite{FLTZ11, FLTZ12} could be
summarized in the following diagram of quasi-equivalence functors.

\begin{theorem}
There is an exact quasi-equivalent functor $\kappa$ as in the
following diagram. Combined with the microlocalization functor $\mu$,
which is also exact and an equivalence (\cite{NZ09,Na09}),
$\mu\circ\kappa$ is a version of homological mirror symmetry --
$\mu\circ\kappa(\cL)$ is isomorphic to the SYZ transformation
of $\cL$, which takes an equivariant line bundle and produces a
Lagrangian in $\Fuk(T^*M_\bR;\Lambda_\bSi)$ as given in \cite{AP,LYZ}.
$$
D^b\Coh_\bT(\cX)\stackrel{\kappa}{\longrightarrow} D^b\Sh_{cc}(M_\bR;\Lambda_\bSi) \stackrel{\mu}{\longrightarrow} \Fuk(T^*M_\bR;\Lambda_\Si).
$$
\end{theorem}

Here we only explain the functor $\kappa$ and $\mu\circ\kappa$ on the
object level. Let $\vl=(l_1,\dots,l_{r'})\in \bZ^{r'}$. We say a line
bundle $\cL$ is $\bQ$-ample if certain postive powers of it is a pullback
of an ample line bundle on the coarse moduli $X$. Given a
$\bT$-equivariant $\bQ$-ample line bundle
$\cL_\vl=\cO_{\cX}(\sum_{i=1}^{r'} l_i \cD_i)$ on $\cX$, one can
associate its moment polytope $$\Delta_\vl=\{u\in M_\bR| \langle u,
b_i \rangle > -l_i\}.$$ The $\bQ$-ampleness ensures that $\Delta_\vl$
is a convex open polytope. The functor $\kappa$ is characterized by
mapping the $\bQ$-ample line bundle $\cL_\vl$ to
$i_{\Delta_\vl!}\omega_{\Delta_\vl}$, where $i_{\Delta_\vl}:
\Delta_\vl \hookrightarrow M_\bR$ is the embedding and
$\omega_{\Delta_\vl}$ is the Verdier dual of the constant sheaf
$\bC_{\Delta_\vl}$ on $\Delta_\vl$. The constructible sheaf $i_{\Delta_\vl*}\bC_{\Delta_\vl}$ is called the
\emph{standard} sheaf over $\Delta_\vl$, while its Verdier dual $i_{\Delta_\vl!}\omega_{\Delta_\vl}$ is called \emph{costandard} sheaf.  The microlocalization functor
$\mu$ takes the standard sheaf $i_{\Delta_\vl*}\bC_{\Delta_\vl}$ to the graph of
$d\log m_{\Delta_\vl}$, where $m_{\Delta_\vl}$ is a function on
${\bar\Delta_\vl}$ with $m_{\Delta_\vl}|_{\partial \Delta_\vl}=0$ and
$m_{\Delta_\vl}|_{\Delta_\vl}>0$. This graph $d\log m_{\Delta_\vl}$ is a  \emph{standard} Lagrangian. The functor also takes the costandard sheaf $i_{\Delta_\vl!}\omega_{\Delta_\vl}$ to a \emph{costandard} Lagrangian which is the graph of $-d\log m_{\Delta_\vl}$.

\begin{figure}[h]
\psfrag{1}{\small $m=x(1-x)$}
\psfrag{2}{\small $\log m = \log x + \log(1-x)$}
\psfrag{3}{\small $d\log m = \frac{dx}{x}-\frac{dx}{1-x}.$}
\includegraphics[scale=0.35]{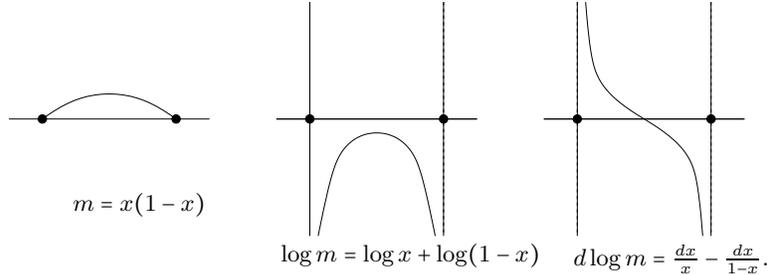}
\caption{The graphs of $m$, $\log m$ and $d\log m$ of an open interval $U\subset \bR$. The graph of $d\log m$ is the standard Lagrangian over $U$. If $\cX=\bP^1$, an ample line bundle $\cO_{\bP^1}(a \cD_1+ b \cD_2)$ corresponds to the costandard sheaf over $(-a,b)$ under $\kappa$. }
\label{fig:standLag}
\end{figure}

\subsection{Landau-Ginzburg B-model and oscillatory integrals}

The symplectic structure on $T^*M_\bR$ only tackles the B-model
information on $\cX$. In order to make predictions on the
Gromov-Witten invariants on $\cX$ one needs to consider the complex
structure on the equivariant mirror Landau-Ginzburg model of $\cX$. In this
subsection, we first define such a Landau-Ginzburg model from the
viewpoint of complex geometry, which should be mirror to the
\emph{A-model} of $\cX$, and then identify it with
$T^*{M_\bR}=N_\bR\times M_\bR$.

We fix an integral basis $e_1,\dots,e_k \in \bL$ and its dual basis
$e_1^\vee,\dots,e_k^\vee$ in $\bL^\vee$. We require that each
$e_a^\vee$ is in $\tNef_\cX$, and $e_{k'}^\vee,\dots,e_k^\vee$ are in
$\sum_{i=r'+1}^r \bR_{\geq 0} D_i$.  Define the \emph{charge vectors}
\[
l^{(a)}=(l_1^{(a)},\dots, l_r^{(a)})\in \bZ^r,\quad
\psi(e_a)=\sum_{i=1}^r l_i^{(a)} \tb_i.
\]
So
\[
D_i=\psi^\vee(D_i^\bT)=\sum_{a=1}^k l_i^{(a)}e_a^\vee,\quad i=1,\dots,r.
\]
Define the Landau-Ginzburg B-model as follows
$$
\cY_q=\{(X_1,\dots,X_r)\in (\bC^*)^r|\prod_{i=1}^r X_i^{l^{(a)}_i}=q_a,\ a=1,\dots,k\}.
$$
Here $q_1,\dots,q_k$ are \emph{complex parameters}. We assume $q_a>0$
for all $a$. Apply the exact functor $\Hom(-,\bC^*)$ to the short exact sequence
\eqref{eqn:NtN} and we get
\[
1\to \Hom(N,\bC^*) \to (\bC^*)^r \stackrel{\mathfrak{p}}{\longrightarrow} \cM=\Hom(\bL,\bC^*)\to 1.
\]
We see that $\cY_q=\mathfrak{p}^{-1}(q)\cong (\bC^*)^k$ is a subtorus in $(\bC^*)^r$. Here $q=(q_1,\dots, q_k)$ are coordinates on $\cM$. For any $\beta\in \bK$, denote $q^\beta=\prod_{a=1}^k q_a^{\langle \beta, e_a^\vee \rangle}$. This notion may involve factional powers as $\beta \in \bK$ in which $\bL$ is a sublattice.

Let $u_1,\dots,u_n$ be the coordinates on $M_\bR$ and $v_1,\dots,v_n$
be the coordinates on $N_\bR$. Let $y_i=-v_i+2\pi\sqrt{-1}u_i$ and
$Y_i=e^{y_i}$. Then $y_1,\dots,y_n$ are complex coordinates on
$\tcY=N_\bR\times M_\bR$, while $Y_1,\dots,Y_n$ are complex coordinates on
$\cY=N_\bR\times M_\bR/M\cong T^*(M_\bR/M)$.

We fix a splitting of the exact sequence \eqref{eqn:NtN} over rational
numbers, i.e. we choose a surjective map $\eta: \tN_\bQ \to \bL_\bQ$ such that
$\eta(\tb_i)=\sum_{a=1}^k \eta_{ia} e_a$ (so $e^\vee_a=\sum_{i=1}^r
\eta_{ia}D_i$) and $\psi\circ\eta=\mathrm{id}$, where $\eta_{ia}\in \bQ$. This splitting identifies
$\cY_q$ with $\cY=\Hom(N,\bC^*)\cong N_\bR\times (M_\bR/M)$ as follows.
\begin{equation}
\label{eqn:B-model-identification}
X_i=q'_iY^{b_i},\quad Y^{b_i}=\prod_{j=1}^n Y_j^{b_{ij}},\quad q_i'=\prod_{a=1}^k q_a^{\eta_{ia}}.
\end{equation}
Here $b_i=(b_{i1},\dots,b_{in})$ is the coordinate of $b_i$ in
$N$. We also identifies $\tcY_q$ with $\tcY=N_\bR \times M_\bR$. We choose
$q_i'>0$ as they may involve fractional powers of $q_a$ in \eqref{eqn:B-model-identification}.

\begin{definition}
The superpotential on $\cY_q$ is
$$
W=\sum_{i=1}^r (t_0+X_i),
$$
where $t_0\in \bC$ is a parameter which plays the role of $H^0$-part on the A-side under the mirror relation. This superpotential depends on
$t_0$ and $q$. We let $q_0=e^{t_0}$. The equivariant superpotential on
$\tcY_q$ is
\[
\tW=W+\sum_{i=1}^r w_i x_i.
\]
Here $x_i=\log X_i$, and the equivariant parameters $w_i\in \bC$.
\end{definition}

Once we choose
\[
x_i=\sum_{a=1}^k \eta_{ia}\log q_a + \sum_{j=1}^n b_{ij} y_j,
\]
the equivariant superpotential $\tW$ is well-defined on $\tcY_q$. It depends on $t_0$,
$\log q_a$ (we can choose them in $\bR$ since all $q_a>0$) and $w_i$.


The following holomorphic form on $\tcY=N_\bR \times M_\bR$
$$
\Omega=\frac{dY_1\dots dY_n}{Y_1\dots Y_n}
$$
is also a holomorphic form on $\tcY_q$ once we identify $\tcY_q$ with
$\tcY$. Let $W_\eta$ and $\tW_\eta$ be the function $W$ and $\tW$ on $\cY$ and $\tcY$ respectively once we fix $\eta$. Consider a partial compactification $(\overline{T^*
  M_\bR})_{\Lambda_\bSi}=T^*M_\bR \cup \Lambda_\bSi^\infty \subset
\overline{T^* M_\bR}$. For any $z>0$ and $$\Xi\in H_n ((\overline{T^*
  M_\bR})_{\Lambda_\bSi}, \Lambda_\bSi^\infty;\bZ),$$ we define the
integral
\begin{equation}
  \label{eqn:oscillatory-integral}
I_\Xi:=\int_\Xi e^{-\frac{\tW_\eta}{z}}\Omega.
\end{equation}
This definition a priori depends on the choice of $\eta$, which gives the identification of
$\tcY_q$ with $\tcY=N_\bR\times M_\bR$. We have the following two propositions.

\begin{proposition}
The integral $I_\Xi$ converges for any $\Xi\in H_n((\overline{T^*
  M_\bR})_{\Lambda_\bSi},\Lambda_\bSi^\infty;\bZ)$ and any choice of $\eta$.
\end{proposition}
\begin{proof}
For each $\tau\in \Si$ where $\tau$ is spanned by $\rho_i$, $i\in
I_\tau'$, the linear Lagrangian
\[
\Lambda_\tau=\bigcup_{\chi\in M_\tau} (-\tau)\times \tau^\perp_\chi,
\]
is characterized by
\begin{gather*}
\mathrm{Re}(y^{b_i})=-\sum_{j=1}^n b_{ij} v_j\geq 0,\quad
\mathrm{Im}(y^{b_i})=2\pi\sqrt{-1}\sum_{j=1}^n b_{ij} u_j
\in 2\pi\sqrt{-1}\bZ,\quad i\in I_\tau';\\
\mathrm{Re}(y^{b_i})=0, \quad i\in I_\tau.
\end{gather*}
We view $T^*M_\bR=N_\bR\times M_\bR$ as the interior part of the unit ball
bundle $\overline {T^*M_\bR}$ under the embedding map $\rho$. Then $\rho_*e^{-\frac{\tW}{z}}$ is an
analytic function on $\rho(T^*M_\bR)$, and since $z>0$ it exponentially decays to zero
near $\Lambda_\tau^\infty$ -- it can be extended to a neighborhood
$\tilde\Lambda_\tau^{\infty}$ of
$\Lambda_\tau^\infty\subset \overline {T^*M_\bR}$ for each $\tau$,
with value $0$ on $\tilde \Lambda_\tau^{\infty}\setminus \rho(T^*M_\bR)$. The
differential form $\rho_*(e^{-\frac{\tW}{z}}\Omega)$ can also be defined on
this neighborhood. Thus denote
\[
(\overline{T^*M_\bR})_{\Lambda_\bSi}=T^*M_\bR \cup \bigcup_{\tau \in
  \Si} \tilde\Lambda_\tau^{\infty}.
\]
The differential form $e^{-\frac{\tW}{z}}\Omega$ is a closed analytic
  form on this $(\overline{T^*M_\bR})_{\Lambda_\bSi}\subset
  \overline{T^*M_\bR}$, and vanishes on the infinity boundary.

Thus the integral can be evaluated as
\[
I_\Xi=\int_{\overline{\rho(\Xi)}}\rho_*(e^{-\frac{\tW}{z}}\Omega).
\]

\end{proof}

\begin{figure}[h]
\psfrag{1}{\small $\Lambda_\bSi$}
\psfrag{2}{\small $\tilde\Lambda_{\si_2}^\infty$}
\psfrag{3}{\small $\tilde\Lambda_{\si_1}^\infty$}
\includegraphics[scale=0.55]{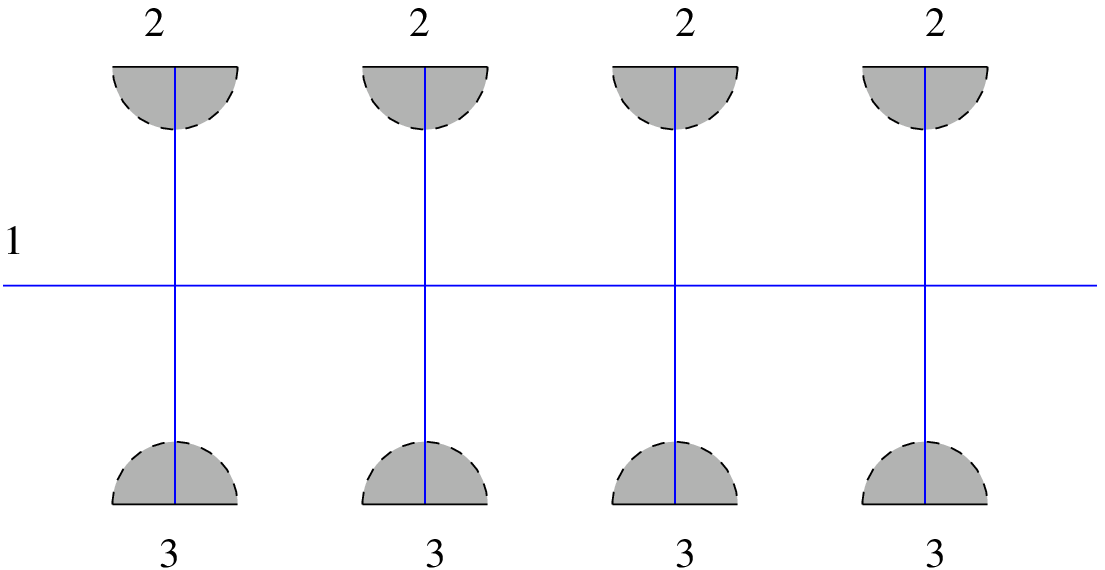}
\caption{$\Lambda_\bSi$ for $\bP^1$ (blue fishbone). Shaded half disks are $\tilde\Lambda_\si^\infty$ on which $\tW$ can be extended. Here two $1$-cones of $\bP^1$ are $\si_1=\bR_{\geq 0}$ and $\si_2=\bR_{\leq 0}$.}
\end{figure}

\begin{proposition}
The integral $I_\Xi$ is does not depend on the choice of $\eta$.
\end{proposition}
\begin{proof}
Let $\eta_1$ and $\eta_2$ be two splittings of the exact sequence
\eqref{eqn:NtN}. Then we have two identifications of $N_\bR\times
M_\bR$ with $\tcY_q$, and two sets of coordiantes $\{y_i^1\}$ and
$\{y_i^2\}$ on $\tcY_q$. By \eqref{eqn:B-model-identification} and the
fact $\eta_{ia}\in\bQ$, they are related by
\[
y_i^1=y_i^2+c_i, \ c_i \in \bR.
\]
The Calabi-Yau form $\Omega$ is invariant under this change of
variables, and the cycle $\Xi$ is also invariant (they are related by
a homotopy $y_i^1=y_i^2+t c_i, t\in [0,1]$, which preserves $\Lambda_\bSi^\infty$).
\end{proof}

Given a sheaf $E\in \Sh_{cc}(M_\bR;\Lambda_\bSi)$, let
\[\CC(E)\in
H_{SS(E)}^0(M_\bR;\pi^{-1} \omega_{M_\bR})\subset
H_{\Lambda_{\bSi}}^0(M_\bR;\pi^{-1}\omega_{M_\bR})\] be its
characteristic cycle. This cycle represents an element in
$H_n((\overline{ T^* M_\bR})_{\Lambda_\bSi},\Lambda_\bSi^\infty;\bZ)$
since $\SS(E)\subset \Lambda^\infty_\bSi$. We still denote it by $\CC(E)$ by a slight abuse of notation. For any coherent sheaf $\cL \in \Coh_\bT(\cX)$, define its mirror cycle to be
\[
\Xi(\cL):=\CC(\kappa(\cL)).
\]

We define the central charge of $E$ as follows.
\begin{definition}[A and B-model central charge]
  \begin{align*}
    I_B(E)&=I_{\CC(E)}=\int_{\CC(E)} e^{-\frac{\tW}{z}} \Omega,\\
    I_A(\cL)&=I_{\Xi(\cL)}=\int_{\CC(\kappa(\cL))} e^{-\frac{\tW}{z}}\Omega.
  \end{align*}
\end{definition}
\begin{proposition}
\label{prop:integral-additivity}
If $E \to F \to G \to E[1]$ is an exact triangle, then
$$
I_B(F)=I_B(E)+I_B(G).
$$
\end{proposition}
\begin{proof}
It follows from $\CC(F)=\CC(E)+\CC(G)$ \cite[Proposition 9.4.5]{KaSc94}.
\end{proof}

We give a description of the characteristic cycle of the constructible sheaf corresponding to a $\bQ$-ample line bundle. Let $\cL_\vl$ be a $\bQ$-ample line bundle. For each cone $\tau\in \Si$ spanned by $b_i\in \tau$, it determines a face $\Delta_\tau(\vl)\subset \Delta_\vl \subset M_\bR$
\[
\Delta_{\tau}(\vl)=\{u\in M_\bR|\langle b_i, u \rangle \geq -l_i, \forall b_i\notin \tau; \quad\langle b_i,u\rangle =-l_i,\forall b_i\in \tau\}.
\]
In particular, each top dimensional cone $\si$ determines a vertex $\uchi_\si(\vl)=\Delta_\si(\vl)$ of $\Delta_\vl=\Delta_{\{0\}}(\vl)\subset M_\bR$. Then $\Xi(\cL_\vl)$ is the following
\begin{equation}
  \label{eqn:char-ample}
  \Xi(\cL_\vl)=\CC(\kappa(\cL_\vl))=\CC(i_{\Delta_{\vl}!}\omega_{\Delta_{\vl}})=\sum_{\tau \in \Si} (-\tau)\times \Delta_{\tau}(\vl).
\end{equation}

\begin{remark}
  \label{rem:SYZ}
  When $\cL_\vl$ is an ample line bundle on a smooth complete toric manifold $\cX$, one can follow the SYZ T-duality procedure as prescribed in \cite{AP,LYZ}. We recall briefly what this procedure is here, and refer to \cite{FLTZ12} for more details in the toric case.

  Let $\cX_0=\cX\setminus \cD_\infty$, where $\cD_\infty=\sum_{i=1}^r \cD_i$ is an anti-canonical divisor. Let $s_\vl$ be a meromorphic section such that $(s_\vl)=\sum_{i=1}^r l_i \cD_i$. It is holomorphic on $\cX_0$. A $\bT_\bR$-invariant hermitian metric $h$ on the line bundle $\cL_\vl$ defines the norm $\|s_\vl\|_h$ as $\bT_\bR$-invariant function  on $\cX_0$. Let $f_{\vl}=-\log \|s_\vl\|_h$. Identify $\cX_0=N\otimes_\bZ \bC^*$, where $v_1,\dots,v_n\in \bR$ and $\theta_1,\dots,\theta_n\in \bR/(2\pi\bZ)$ are coordinates on $\cX_0$ such that they give complex coordinates $e^{v_i+\sqrt{-1}\theta_i}$. The $\bT_\bR$-invariance of $f_\vl$ implies $f_\vl$ is a function of $v$ only. Define the SYZ dual brane by
  \[
  \SYZ(\cL_\vl)=\{(u,v)\in M_\bR\times N_\bR: u_j=\frac{\partial f_{\vl,h}}{\partial v_j},j=1,\dots,n\}\subset \tcY.
  \]
  The resulting submanifold, as shown in \cite{FLTZ12}, is a costandard Lagrangian on the open set $\Delta_\vl$ when $\cL_\vl$ is ample. The results are isomorphic in the Fukaya category for different choices of $h$. Moreover, for any costandard sheaf $\mathcal F$, the result of \cite{SV} states the characteristic cycle is the conical limit of the costandard Lagrangian $\mu(\cF)$
  \[
  \CC(\cF)=\lim_{\epsilon\to 0}\epsilon \mu(\cF).
  \]
  The integral on $\epsilon \SYZ(\cL_\vl)$
  \[
  \int_{\epsilon \SYZ(\cL_\vl)} e^{-\frac{\tW}{z}}\Omega
  \]
  remains unchanged as $\epsilon\to 0$ since $\epsilon\SYZ(\cL_\vl)=\epsilon \mu\circ\kappa(\cL_\vl)$ corresponds to the same cycle in $H_n((\overline {T^* M_\bR})_{\Lambda_\bSi},\Lambda_\bSi^\infty;\bZ)$. So
  \[
  \int_{\SYZ(\cL_\vl)} e^{-\frac{\tW}{z}}\Omega=I_A(\cL_\vl).
  \]
\end{remark}

\section{Oscillatory integrals and Picard-Fuchs equations}
\label{sec:oscillatory}
\subsection{GKZ system}

Following Iritani \cite{Ir09}, define the $H^*_{\orb,\tbT}(\cX;\bC)$-valued equivariant $I$-function as below.
\begin{definition}
$$
I^\tbT(q_0,q,z)=e^{(t_0+\sum_{a=1}^k t_a p_a^\tbT)/z} \sum_{d\in \bK_{\eff}}q^d \frac{\prod_{\{m\}=\{\langle D_i, d \rangle \}, m\leq 0} (\bar D_i^\tbT+mz)}{\prod_{\{m\}=\{\langle D_i, d \rangle \}, m\leq \langle D_i, d \rangle} (\bar D_i^\tbT+mz)} \one_{v(d)},
$$
where $p_a^\tbT=(c_1)_\tbT(\cL_{e_a^\vee})$, $t_a=\log q_a$ for $a=0,\dots, k$.
\end{definition}
The pullbacks of $I^\tbT$ to torus fixed points in $\IX$ are functions such that the oscillatory integral $I_\Xi$ is a linear combination of them. Iritani has shown this in \cite{Ir09}, with slightly different symbols and set of parameters. We quote his arguments and follow his notions, which lead to Proposition \ref{prop:linear-combination} in our notions. In the following text, $\uw_1,\dots,\uw_r$ form a basis of $H^*_{\tbT}(\mathrm{pt})$ (A-model), and let $w_1,\dots,w_r,\lambda_1,\dots,\lambda_r$ be complex numbers (B-model). Notice that $\uw_i=-(c_1)_\tbT (\cO_{\bP^\infty}(-1))$. Here we fix $H^*_{\tbT}(\IX)$-valued function ($\rho_a=\sum_{i=1}^r l_i^{(a)}$)
\[
H^\tbT(q,z)=(-1)^n z^{-\frac{\uw_1+\dots+\uw_r}{2\pi\sqrt{-1}}} \sum_{d\in \bK_\eff}  \frac{\prod_{a=1}^k(\frac{q_a}{z^{\rho_a}})^{\frac{p^\tbT_a}{2\pi\sqrt{-1}}} q^d \one_{\inv(v(d))}}{z^{\langle D_1+\dots+D_r, d\rangle}\prod_{i=1}^r\Gamma(1+\langle D_i, d\rangle +\frac{\bar D^\tbT_i}{2\pi\sqrt{-1}})}.
\]
Let $\iota_{\si,v}:\fp_{\si,v}\hookrightarrow \cI\cX$ be a $\bT$-fixed point. Then define
$$
H_{\si,v}(q,z)=\iota_{\si,v}^* H^\tbT(q,z),\quad I_{\si,v}(q,z)=\iota_{\si,v}^* I^\tbT(q_0,q,z)\in \tM.
$$
Comparing with $H_{\si,\inv(v)}|_{\uw_i=2\pi\sqrt{-1}w_i/z}$ with $I_{\si,v}|_{\uw_i=w_i}$ we find that $$H_{\si,\inv(v)}|_{\uw_i=2\pi\sqrt{-1}w_i/z}=c_{\si,v}(z,t_0,w_i)I_{\si,v}\vert_{\uw_i=w_i}$$ for $z>0,$ where  $t_0\in \bC$ and generic $w_i\in \bC$.

In \cite[Lemma 4.19]{Ir09}, because $H_{\si,v}$ are fundamental solutions to a GKZ system, the following integral can be expressed as a linear combination where $z>0$, $t_0\in \bC$, generic $w_i\in \bC$ and $0<q_a<\epsilon$ for some $\epsilon>0$
\[
\int_\Xi e^{-\frac{W}{z}-\frac{\lambda_1+\dots+\lambda_r}{2\pi\sqrt{-1}}}\Omega=e^{\frac{-t_0}{z}}\sum_{(\si,v)} b_{\si,v}(\lambda_i) H_{\si,v}(q,-z)\vert_{\uw_i=\lambda_i},
\]
Setting $\lambda_i=\frac{2\pi\sqrt{-1}w_i}{z}$, we have the following proposition, where $h_{\si,v}(z,t_0,w_i)=c_{\si,v}(-z,t_0,w_i) b_{\si,\inv(v)}$.
\begin{proposition}
  \label{prop:linear-combination}
  For $z>0, t_0\in \bC$, generic $w_i\in\bC$ and $0<q_a<\epsilon$ with some $\epsilon>0$,
$$I_\Xi=\sum_{(\si,v)}h_{\si,v}(z,t_0,w_i)I_{\si,v}(q,-z)|_{\uw_i=w_i}.$$
\end{proposition}

\subsection{Fixing the coefficients}
For any $\bT$-equivariant line bundle $\cL$ on $\cX$, define $\cL_v=(\pr^*\cL)\vert_{\cX_v}$.  The stabilizer of $\cX_v$ acts on the orbifold line bundle $\cL_v$ on $\cX_v$ by $\exp(2\pi \sqrt{-1}f)$, where the rational number $0\leq f <1$, called the \emph{age of $\cL_v$ along $\cX_v$}. We denote the age of $\cL_v$ along $\cX_v$ by $\age_v(\cL)$.  We define the following characteristic classes in $H^*_{\tbT}(\cIX)$ ($\cL$ is also  $\tbT$-equivariant since the action of $\tbT$ factors through $\bT$)

\begin{align}
  \label{eqn:ch}
\widetilde{\mathrm{ch}}_z(\cL)&=\bigoplus_{v\in \mathrm{Box}} e^{2\pi\sqrt{-1}(-\frac{(c_1)_{\tbT} (\cL_v)}{z}+\age_{v}(\cL))},\\
  \label{eqn:gamma}
\tGamma_z(\cL)&=\bigoplus_{v\in \mathrm{Box}} z^{\frac{(c_1)_{\tbT}(\cL_v)}{z}+1-\age_v(\cL)}\Gamma(\frac{(c_1)_{\tbT}(\cL_v)}{z}+1-\age_v(\cL)).
\end{align}We understand these classes as series of cohomology classes. In particular we expand the Gamma function at $1-\age_v(\cL)> 0$. We extend these classes to any $\cE\in D^b\Coh_\bT(\cX)$ via
\begin{align}
  \label{eqn:class-additivity}
  \tch_z(\cE)&=\tch_z(\cE_1)+\tch_z(\cE_2),\\
  \tGamma_z(\cE)&=\tGamma_z(\cE_1)\tGamma_z(\cE_2), \nonumber
\end{align}
for any exact triangle $\cE_1\to \cE\to \cE_2\to\cE_1[1]$.

Let $\sigma$ be a top-dimensional cone, and write $b_j=\sum_{b_i\in \sigma} s_{ij} b_i$. Then we have
$
D_i=-\sum_{b_j\notin \sigma}{s_{ij}}D_j.
$ The $\tbT$-equivariant Chern roots of $\iota_{\sigma,v}^* T\cX$ are $-\tuw_i$ for $b_i\in \sigma$ where
$$
\tuw_i=\uw_i+\sum_{b_j\notin \sigma} s_{ij}\uw_j.
$$
We also define complex numbers
\[
\tw_i=w_i+\sum_{b_j\notin \si}s_{ij}w_j.
\]

Let $\cL_\vl$ be an ample line bundle, $\si$ be a top dimensional cone, and $v_0\in \si$. Define the piecewise Lagrangian for $t\geq 0$
\begin{equation}
  \label{eqn:ample-char-shifted}
\Xi_{\si,t}(\cL_\vl)=\Xi(\cL_\vl)+ t v_0=\sum_{\tau \in \Si} (-\tau+tv_0)\times \Delta_\tau(\vl).
\end{equation}
This Lagrangian is not conical for $t\neq 0$. They are all homotopic to each other, and corresponds to the same element in $H_n((\overline{ T^*M_\bR})_{\Lambda_\bSi},\Lambda_\bSi^\infty;\bZ)$.

For a given top dimensional cone $\si$, there is a uniquely determined $\eta_\si$ such that
$\eta_{\si}(\tb_i)=0$, i.e. $\eta_{\si,ia}=0$ for $b_i\in \si$. By the fact that each $e_a^\vee \in \tNef_\si$ and $e_a^\vee=\sum_{i=1}^r \eta_{\si,ia} D_i$, we know each $\eta_{\si,ia}>0$ for $b_i\notin \si$. So each $q'_{\eta_\si,i}=\prod_{a=1}^k q_a^{\eta_{\si,ia}}$ for $b_i\notin \si$, and $\lim_{q\to 0} q'_{\eta_\si}=0$. This limit is called the large radius limit when $\cX$ is a smooth manifold.

For this $\si$ and $\eta_\si$, we define
\[
X^\si_i=X_i=Y^{b_i},\ b_i\in \si;\quad X^\si_i=Y^{b_i},\ b_i\notin \si.
\]
The superpotential and the equivariantly perturbed superpotential are
\begin{gather*}
  W_\eta=\sum_{b_i\in \si}X_i+\sum_{b_i\notin \si} X_i=\sum_{b_i\in \si} X_i^\si+ \sum_{b_i\notin \si} q'_{\eta_\si,i} X_i^\si,\\
  \tW_\eta=W_\eta+\sum_{i=1}^r w_i x_i.
\end{gather*}

We list a simple fact as a lemma, which we will use several times.
\begin{lemma}
\label{lemm:twi}
For a top dimensional cone $\si$ and any given $A\gg 0$, we are able to choose $w_1,\dots, w_r$ in an open region in $\bC^r$ such that $\mathrm{Re} \tw_i<-A,\ b_i\in \si$.
\end{lemma}
\begin{proof}
Let $\mathrm{Re}w_i\ll 0$ for $b_i\in \si$ while keeping $w_i$ bounded for $b_i\notin \si$.
\end{proof}

The map $\nu$ defined in Equation \eqref{eqn:v} identifies $\bK_{\si}/\bL$ with $\Boxs$. Since we have fixed an integral basis $e_1,\dots, e_k$ of $\bL$, this $\nu$ identifies $\{\sum_{i=1}^k t_i e_i|0\leq t_i <1\}\cap \bK_\si$ with $\Boxs$. In the rest of this section we regard any $v\in \Boxs$ as an element in $\bK_{\eff,\si}$.

\begin{proposition}
For any line bundle $\mathcal L_{\vl}$, the coefficients in the decomposition of $I_A(\cL_\vl)$ in Proposition \ref{prop:linear-combination} are (notice that $h_{\si,v}$ do no depend on $t_0$)
$$
h_{\si,v}(z,w_i)=\frac{\iota_{\sigma,v}^*\left( \inv^*(\widetilde\Gamma_{z}(T\cX)\widetilde{\mathrm{ch}}_{z}(\cL_\vl))\right)\vert_{\uw_i=w_i}}{|G_\sigma|(\iota^*_{\si,v}e_{\tbT}(T\cIX))|_{\uw_i=w_i}}.
$$
\end{proposition}
\begin{proof}
We just need to prove this statement for ample line bundle $\mathcal L_\vl$. The proposition then follows from the additivity of $\tch$ (Equation  \ref{eqn:class-additivity}) and characteristic cycles (Proposition \ref{prop:integral-additivity}), together with the fact that $\bT$-equivariant $\bQ$-ample line bundles generate the equivariant $K$-group and the derived category $D^b\Coh_\bT(\cX)$. Since the choice of $\eta$ does not affect the integral $I_{\Xi(\cL_\vl)}$, we simply need to consider \[\int_{\Xi(\cL_\vl)}e^{-\frac{\tW_{\eta_\si}}{z}}\Omega.\]

In this proof, the symbols $O()$ and $o()$ concern the limit $q\to 0$.
We have
\begin{align*}
 &|G_\sigma| e^{\frac{t_0}{z}}\prod_{j\notin \sigma}(q'_{\eta_\si,j})^{\frac{w_j}{z}}\int_{\Xi(\cL_\vl)} e^{-\frac{\tW_\eta}{z}} \Omega=\int_{\Xi(\cL_\vl)} e^{-\frac{\sum_{b_i\in\sigma} X_i}{z}-\frac{\sum_{b_i\notin \sigma} X_i}{z}}\prod_{i=1}^r {(X^\si_i)^{-\frac{w_i}{z}}} \prod_{b_i\in \sigma} \frac{dX^\si_i}{X^\si_i}\\
 =&\int_{\Xi_{\si,t}(\cL_\vl)} e^{-\frac{\sum_{b_i\in\sigma} X_i}{z}-\frac{\sum_{b_i\notin \sigma} X_i}{z}}\prod_{i=1}^r {(X^\si_i)^{-\frac{w_i}{z}}} \prod_{b_i\in \sigma} \frac{dX^\si_i}{X^\si_i}.
 \end{align*}
Suppose $N>0$, and let $f(q_1,\dots,q_a)$ be a function of $q\in (\bR_{> 0})^r$. We say $f(q)=O_N$ if $f(tq_1,\dots,tq_a)=O(t^N)$ as $t\to 0$. Consider the expansion
\begin{align*}
&e^{-\frac{\sum_{b_i\notin \sigma} X^\si_i}{z}}\prod_{i=1}^r (X^\si_i)^{-\frac{w_i}{z}-1}=\sum_{\beta \in \bK_{\eff,\si}} \prod_{b_i\not\in \si}\frac{q^\beta}{\langle \beta, D_i \rangle !(-z)^{\langle \beta,D_i\rangle}} \prod_{b_i\in \si}(X^\si_i)^{-\langle \beta, D_i \rangle -\frac{\tw_i}{z}-1}\\
=&\sum_{\beta \in \bK_{\eff,\si},|\beta|\leq N} \prod_{b_i\not\in \si}\frac{q^\beta}{\langle \beta, D_i \rangle !(-z)^{\langle \beta,D_i\rangle}} \prod_{b_i\in \si}(X^\si_i)^{-\langle \beta, D_i \rangle -\frac{\tw_i}{z}-1}+f(q,X^\si,w_i).
\end{align*}
Here the remaining terms $f(q,X^\si,w_i)$, and $|\beta|\leq N$ denotes the condition $\sum_{i=1}^a \langle \beta, e^\vee _a\rangle\leq N$.

We first choose an $N\gg 0$, and by Lemma \ref{lemm:twi}, we choose $w_i$ such that $\mathrm{Re}(-\frac{\tw_i}{z})>N+1$. We have
\begin{align*}
&\int_{\Xi_{\si,t}(\cL_\vl)} e^{-\frac{\sum_{b_i\in\sigma} X_i}{z}-\frac{\sum_{b_i\notin \sigma} X_i}{z}}\prod_{i=1}^r {(X^\si_i)^{-\frac{w_i}{z}}} \prod_{b_i\in \sigma} \frac{dX^\si_i}{X^\si_i}\\
=&\int_{\Xi_{\si,t}(\cL_\vl)} e^{-\frac{\sum_{b_i\in\sigma} X^\si_i}{z}}(\sum_{\beta \in \Boxs} \frac{q^\beta\prod_{b_i\in \si}(X^\si_i)^{-\langle \beta, D_i \rangle -\frac{\tw_i}{z}-1}}{\prod_{b_i\not\in \si}\langle \beta, D_i \rangle !(-z)^{\langle \beta,D_i\rangle}})\prod_{b_i\in \si}dX^\si_i \\ &\qquad \qquad   +\int_{\Xi_{\si,t}(\cL_\vl)}e^{-\frac{\sum_{b_i\in\si}X^\si_i}{z}} f(q,X^\si,w_i)\prod_{b_i\in \si}dX^\si_i\\
=&\sum_{\beta\in \bK_{\eff,\si}, |\beta| \leq N} \prod_{b_i\notin \sigma}\frac{q^\beta}{\langle \beta, D_i\rangle!(-z)^{\langle \beta, D_i\rangle}} \prod_{b_i\in \sigma} \frac{\Gamma(-\langle \beta, D_i \rangle - \frac{\tw_i}{z})}{z^{\langle \beta, D_i \rangle +\frac{\tw_i}{z}}}e^{-2\pi\sqrt{-1}(-\langle \beta,D_i \rangle-\frac{\tw_i}{z})l_i}+O_N.
\end{align*}
In the computation above, the first two expressions do not depend on the value of $t$. So taking $t\to \infty$, the first integral on the second expression becomes an integral on $N_\bR\times \uchi_\si(\vl)$ -- the Gamma function evaluates this integral. While the second integral on the second expression produces the $O_N$ term since $f(q,X^\si,w_i)=O_N$.

We compute the pull-back of the $I$-function to $\fp_{\si',v}$ where $v\in \mathrm{Box}(\si')$ for any top dimensional cone $\si'$. Note that
\begin{gather*}
\iota_{\si,v}^* \bar D_i^\tbT=-\tuw_i,\quad b_i\in \si,
\end{gather*}
while
\[
\iota_{\si',v}^* \left(\prod_{a=1}^k q_a^{\frac{p_a^\tbT}{z}}\right)=\prod_{b_j\notin \si'} q_{\eta_{\si'},j}'^{\frac{w_j}{z}},\qquad
\prod_{b_i\notin \si}q_{\eta_\si,i}'^{\frac{w_i}{z}}=\prod_{b_j\notin \si\cup \si'}q_{\eta_{\si'},j}'^{\frac{w_j}{z}}\prod_{b_j\in (\si\setminus \si')}q_{\eta_{\si'},j}'^{-\frac{\sum_{b_i\notin \si} s_{ij}w_i}{z}}.
\]

We compute:
\begin{align*}
&I_{\sigma,v}(q,-z)|_{\uw_i=w_i}\\
=&e^{-\frac{t_0}{z}}\prod_{b_j\notin\sigma}(q'_{\eta_\sigma,j})^{-\frac{w_j}{z}}\left(\frac{q^v}{\prod_{b_i\notin \sigma}\langle v,D_i \rangle!(-z)^{\langle v,D_i \rangle}}\prod_{b_i\in\sigma}\frac{1}{(-z)^{\lceil \langle v,D_i \rangle \rceil }(\langle v,D_i\rangle +\frac{\tw_i}{z})_{\lceil \langle v,D_i \rangle \rceil}}+q^v O_1\right),\\
&I_{\sigma',v}(q,-z)|_{\uw_i=w_i}=e^{-\frac{t_0}{z}}\prod_{b_j\notin\sigma'}(q'_{\eta_{\sigma'},j})^{-\frac{w_j}{z}}\left(O(q^v)\right),\quad \si'\neq \si.
\end{align*}
Here we adopt the Pochhammer symbol $(a)_b=\Gamma(a)/\Gamma(a-b+1)$. Then
\begin{align*}
&|G_{\si}|e^{\frac{t_0}{z}}\prod_{b_i\notin \si} (q'_{\eta_\si,i})^{\frac{w_i}{z}}I_{\si,v}(q,-z)|_{\uw_i=w_i}\\
=&\frac{q^v}{\prod_{b_i\notin \sigma}\langle v,D_i \rangle!(-z)^{\langle v,D_i \rangle}}\prod_{b_i\in\sigma}\frac{1}{(-z)^{\lceil \langle v,D_i\rangle\rceil}(\langle v,D_i\rangle +\frac{-\tw_i}{-z})_{\lceil v,D_i \rceil}}+q^vO_1,\\
&|G_{\si}|e^{\frac{t_0}{z}}\prod_{b_i\notin \si} (q'_{\eta_\si,i})^{\frac{w_i}{z}}I_{\si',v}(q,-z)|_{\uw_i=w_i}=\prod_{b_j\in (\si\setminus \si')}q_{\eta_{\si'},j}'^{-\frac{w_j+\sum_{b_i\notin\si}{s_{ij}w_i}}{z}}O(1).
\end{align*}
We further choose $w_1,\dots, w_r$ such that $\mathrm{Re}\tw_i<-\max\{A,N\}$ for each $b_i \in \si$ by Lemma \ref{lemm:twi}. Here $A=A(N)$ is sufficiently large such that for any $\si'\neq \si$ and $b_i\in (\si\setminus \si')$, $q_{\eta_{\si'},i}'^{-\frac{\tw_i}{z}}=q^\beta$, and $|\beta|>N.$ Comparing with the coefficients of $I_{\sigma,v}(q,-z)$, we find that
\begin{align*}
h_{\sigma,v}(z,w_i)
&=\frac{1}{|G_\si|}\prod_{b_i\in \sigma} z^{-\frac{\tw_i}{z}+\{-\langle v,D_i\rangle\}}\Gamma(\frac{-\tw_i}{z}+\{-\langle v,D_i\rangle\})e^{-2\pi\sqrt{-1}(-\langle v,D_i \rangle -\frac{\tw_i}{z})l_i} \\
&=\frac{\iota_{\sigma,v}^*\left( \inv^*(\widetilde\Gamma_{z}(T\cX)\widetilde{\mathrm{ch}}_{z}(\cL_\vl))\right)\vert_{\uw_i=w_i}}{|G_\sigma|(\iota^*_{\si,v}e_{\tbT}(T\cIX))|_{\uw_i=w_i}}.
\end{align*}
Here $(c_1)_{\tbT}(\cL_\vl)=-\sum_{i\in\sigma} l_i \tuw_i$, and $\age_v(\cL_\vl)=\sum_{b_i\in \sigma} \{-\langle v,D_i\rangle \}l_i$.
\end{proof}

\begin{figure}[h]
\psfrag{M}{\small $M_\bR\cong \bR$}
\psfrag{N}{\small $N_\bR\cong \bR$}
\psfrag{a}{\small $-\infty -2l_1\pi\sqrt{-1}$}
\psfrag{b}{\small $\infty+2l_2\pi\sqrt{-1}$}
\includegraphics[scale=0.37]{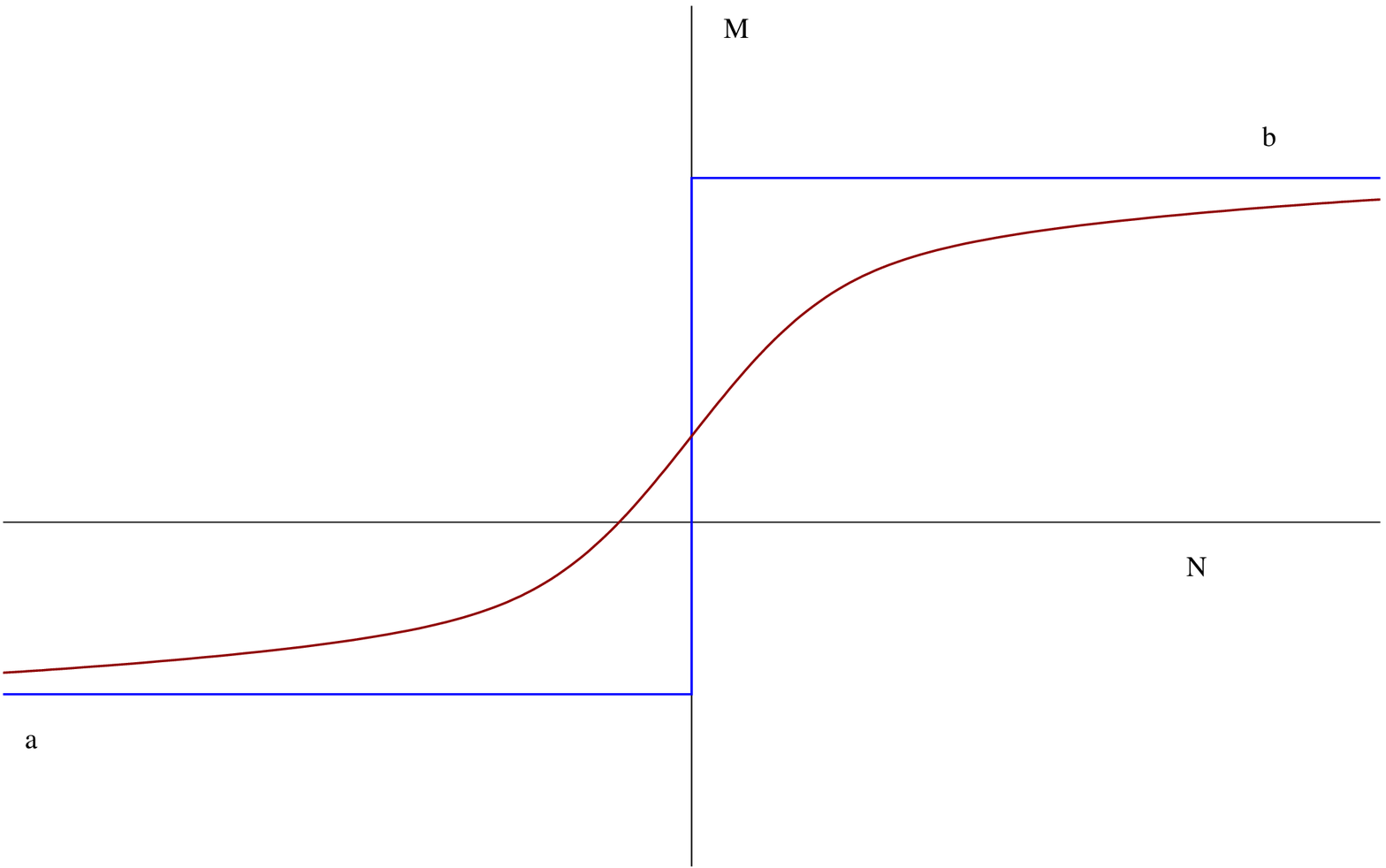}
\caption{The characteristic cycle (blue) $\CC(\cO_{\bP^1}(l_1 \cD_1 + l_2 \cD_2))$ and the SYZ dual of the same line bundle (red). They live in $T^*{M_\bR}\cong N_\bR \times M_\bR$. }
\label{fig:p1-cycle}
\end{figure}

\begin{example}[\cite{FLZ15}]
$\cX=\bP^1$, $\si_1=\si=\bR_{\geq 0},\si_2=\bR_{\leq 0}$, $\tW_{\eta_\si}=X_1+\frac{q}{X_1}+w_1 \log X_1 + w_2 \log \frac{q}{X_1}.$ Let $\cL_{\vl}=\cO(l_1 \cD_1+ l_2 \cD_2)$. When $q\to 0$,
\begin{align*}
e^{t_0/z}q^{w_2/z} \int_{\Xi(\cL_\vl)} e^{-\tW_{\eta_\si}/z}\Omega&= \int_{\Xi(\cL_\vl)} e^{-\frac{X_1}{z}} e^{-\frac{q}{z X_1}} X_1^{(w_1-w_2)/z} d X_1/X_1\\
&\to z^{\frac{w_2-w_1}{z}}\Gamma(\frac{w_2-w_1}{z}) e^{\frac{2\pi\sqrt{-1}l_1(w_2-w_1)}{z}}=\iota_\sigma^* (\tGamma_z(T\cX) \tch(\cL_\vl))_{\uw_i=w_i}.
\end{align*}
Thus \[h_{\si_1}=z^{\frac{w_2-w_1}{z}}\Gamma(\frac{w_2-w_1}{z}) e^{\frac{2\pi\sqrt{-1}l_1(w_2-w_1)}{z}}.\]
Similarly
\[
h_{\si_2}=z^{\frac{w_1-w_2}{z}}\Gamma(\frac{w_1-w_2}{z}) e^{\frac{2\pi\sqrt{-1}l_2(w_1-w_2)}{z}}.
\]
Comparing to the notion of \cite{FLZ15}, $\cD_1=p_2$ and $\cD_2=p_1$ for $p_1,p_2$ defined in \cite{FLZ15}.
\end{example}

\begin{example}
$\cX$ is a semi-positive smooth projective toric variety.
\begin{gather*}
\lim_{q\to 0} e^{t_0/z}\prod_{j\notin \sigma}(q_{\eta_\sigma,j}')^{\frac{w_j}{z}}\int_{\Xi(\cL_\vl)} e^{-\frac{\tW_{\eta_\si}}{z}}\Omega= \prod_{i\in \sigma}  z^{- \frac{\tw_i}{z}}\Gamma(- \frac{\tw_i}{z})e^{2\pi\sqrt{-1}\frac{\tw_i}{z}l_i}= \tGamma_z(T\cX)\tch_z(\cL),\\
\lim_{q\to 0} \iota_\sigma^* \left (e^{t_0/z}\prod_{j\notin \sigma} q_{\eta_\sigma,j}'^{\frac{p_j^\tbT}{z}} I^\tbT(q_0,q,-z)\right)=1,\\
\int_{\Xi(\cL_\vl)} e^{-\frac{\tW}{z}}\Omega= \sum_\sigma \iota_{\sigma}^* (\tGamma_z(TX) \tch_z(\cL) I^\tbT(q_0,q,-z))|_{\uw_i=w_i}.
\end{gather*}
\end{example}

\section{Mirror theorem and Gromov-Witten potentials}
\label{sec:mirror}

\begin{definition}
  \label{def:gw-potential}
  Let $\cX$ be a complete toric orbifold. We define genus $g$, degree $d\in H_2(\cX;\bZ)$, $\tbT$-equivariant descendant Gromov-Witten invariants of $\cX$ as
  \[
    \langle \tau_{a_1}(\gamma_1)\dots \tau_{a_n}(\gamma_n)\rangle_{g,n,d}^{\cX}=\langle \gamma_1 \psi_1^{a_1}\dots \gamma_n\psi_n^{a_n}\rangle_{g,n,d}^{\cX}=\int_{[\overline{\cM}_{g,n}(\cX;d)]^\vir}\prod_{j=1}^n \psi_j^{a_j}\ev^*_j(\gamma_j)\in \bC[\uw],
  \]
  where $\gamma_i\in H^*_{\CR,\tbT}(\cX)$ and $\ev_j:\overline{\cM}(\cX;d)\to \cIX$ is the $j$-th evaluation map. Let $\btau\in H^{\leq 2}_{\CR,\tbT}(\cX;\bC)$. We also define
  \begin{align*}
   \llangle \tau_{a_1}(\gamma_1),\dots,\tau_{a_n}(\gamma_n)\rrangle_{g,n}^{\cX}=\llangle \gamma_1 \psi_1^{a_1},\dots,\gamma_n\psi_n^{a_n}\rrangle_{g,n}^{\cX}\\
   =\sum_{d\geq 0} \sum_{\ell=0}^\infty \frac{1}{\ell!}\langle \tau_{a_1}(\gamma_1),\dots, \tau_{a_n}(\gamma_n),\underbrace{\tau_0(\btau),\dots,\tau_0(\btau)}_{\text{$\ell$ times}}\rangle_{g,n+\ell,d}^{\cX}.
 \end{align*}
\end{definition}
We do not use Novikov variables since the convergence issue regarding $e^{\btau}$ is resolved in \cites{Ir07,CoIr15}, or after invoking the mirror theorem and the oscillatory integral expression of the $I$-function. For the semi-positive $\cX$, we quote the equivariant mirror theorem of toric stacks as below.
\begin{theorem}[Coates--Corti--Iritani--Tseng \cite{CCIT15}, Cheong--Ciocan-Fontanine--Kim \cite{CCK15}]
Let $\phi_\alpha$ be a basis of $H^*_{\orb,\tbT}(\cX;\bC)$ and $\phi^\alpha$ be its dual.
$$\llangle\frac{\phi_\alpha}{z(z-\psi)}\rrangle_{0,1}^\cX\phi^\alpha=I^\tbT(q_0,q,z),$$
under the mirror map
$$
\btau=\btau(q_0,q,z).
$$
Here we understand $\llangle \frac{\phi_\alpha}{z(z-\psi)}\rrangle_{0,1}^\cX$ as a power series in $z^{-1}$. The mirror map is the coefficient of the $z^{-1}$-term in the expansion of $I^\tbT$:
\[
I^\tbT(q_0,q,z)=\one+\frac{\btau(q_0,q)}{z}+o(z^{-1}).
\]
\end{theorem}
Define $$\phi_{\sigma,v}=\frac{\iota_{\sigma,v *}1}{e_\tbT(T_{\fp_{\sigma,v}}\cIX)}.$$
By the equivariant Atiyah-Bott localization, for any class $\alpha \in H^*_{\orb,\tbT}(\cX;\bC)$,
$$
\alpha=\sum_{\sigma,v}(\iota_{\sigma,v}^*\alpha) \phi_{\sigma,v}.
$$
Let $\phi^{\sigma,v}$ be the dual basis to $\phi_{\sigma,v}$, and we have
$$
\phi^{\sigma,\inv(v)}=|G_\sigma|e_\tbT(T_{\fp_{\sigma,v}}\cIX)\phi_{\sigma,v}.
$$

\begin{theorem}
\label{thm:main}
For any $\bT$-equivariant coherent sheaf $\cE$ on $\cX$, we define $\kappa(\cE)=\tGamma_z(T\cX)\tch_z(\cE)$. Then for generic $w_i\in \bC, z>0$ and $0<q_a<\epsilon$ where some small $\epsilon>0$
\begin{align*}
I_A(\cE)=\int_{\Xi(\cE)}e^{-\frac{\tW}{z}}\Omega= \llangle\frac{\kappa(\cE)}{z(z+\psi)} \rrangle_{0,1}^{\cX}\vert_{\uw_i=w_i}.
\end{align*}
\end{theorem}
\begin{proof}
We only need to prove this theorem when $\cE=\cL_\vl$ since line bundles generates the equivariant K-group. Since $$I^\tbT(q_0,q,-z)=\sum_{\sigma,v}(\iota_{\sigma,v}^* I^\tbT(q_0,q,-z))\phi_{\sigma,v},$$
we know that
$$
\llangle \frac{\phi^{\sigma,v}}{z(z+\psi)}\rrangle_{0,1}^\cX= \iota_{\sigma,v}^* I^\tbT(q_0,q,-z).
$$
Then
\begin{align*}
I_A(\cL_\vl)&=\sum_{\sigma,v} h_{\sigma,v}(z,w_i)\iota_{\sigma,v}^*I^\tbT(q_0,q,-z)|_{\uw_i=w_i}\\
&=\sum_{\sigma,v} \llangle \frac{\iota_{\sigma,v}^*(\inv^*(\tGamma_z(T\cX)\tch_z(\cL_\vl)))|_{\uw_i=w_i}\phi^{\sigma,v}}{|G_\sigma|(\iota_{\si,v}^*(e_\tbT(T\cIX)))_{\uw_i=w_i} z(z+\psi)}\rrangle_{0,1}^\cX|_{\uw_i=w_i}\\
&=\sum_{\sigma,v} \llangle \frac{\iota_{\sigma,v}^* (\inv^*(\tGamma_z(T\cX)\tch_z(\cL_\vl)))|_{\uw_i=w_i} \phi_{\sigma,\inv(v)}e_\tbT(T_{\fp_{\sigma,\inv(v)}}\cIX)}{(\iota_{\si,v}^*(e_\tbT(T\cIX)))_{\uw_i=w_i}z(z+\psi)}\rrangle_{0,1}^\cX \vert_{\uw_i=w_i}\\
&=\llangle\frac{\tGamma_z(T\cX)\tch_z(\cL_\vl)}{z(z+\psi)} \rrangle_{0,1}^\cX \vert_{\uw_i=w_i}.
\end{align*}
\end{proof}
\begin{remark}
\label{rmk:main}
By Remark \ref{rem:SYZ}, we have the following when $\cL$ is an ample line bundle on smooth $\cX$.
$$
\int_{\SYZ(\cL)} e^{-\tW/z}\Omega=\llangle\frac{\kappa(\cL)}{z(z+\psi)} \rrangle_{0,1}^\cX\vert_{\uw_i=w_i}.
$$
\end{remark}

\end{document}